\documentclass[12pt]{amsart}
\usepackage{amsthm}
\usepackage{amsmath}
\usepackage[margin=1in]{geometry}
\usepackage{amssymb,verbatim}
\usepackage{tikz-cd}

\theoremstyle{plain}
\newtheorem{theorem}{Theorem}[section]          
\newtheorem{corollary}[theorem]{Corollary}     \newtheorem*{nonumbercorollary}{Corollary}
\newtheorem{lemma}[theorem]{Lemma}              
\newtheorem{proposition}[theorem]{Proposition}

\newtheorem{main}{Theorem}
\theoremstyle{definition}
\newtheorem{definition}[theorem]{Definition}

\newtheorem{notation}[theorem]{Notation} 
\newtheorem{remark}[theorem]{Remark}\newtheorem*{nonumberremark}{Remark}

\numberwithin{table}{section}

\newcommand{\Z}{\mathbb{Z}}\newcommand{\Q}{\mathbb{Q}}

\newcommand{\s}{\mathbb{S}}

\DeclareMathOperator{\codim}{codim}

\DeclareMathOperator{\tr}{tr}
\DeclareMathOperator{\diag}{diag}

\newcommand{\tensor}{\otimes}
\newcommand{\embedded}{\hookrightarrow}

\title{Obstructions to free actions on Bazaikin spaces}
\author{Elahe Khalili Samani}
\email{ekhalili@syr.edu}
\address{Department of Mathematics, Syracuse University, Syracuse, New York 13244}

\begin{document}

\begin{abstract}
Apart from spheres and an infinite family of manifolds in dimension seven, Bazaikin spaces are the only known examples 
of simply connected Riemannian manifolds with positive sectional curvature in odd dimensions. 
We consider positively curved Riemannian manifolds whose universal covers have the same cohomology as Bazaikin spaces 
and prove structural results for the fundamental group in the presence of torus symmetry.\vspace{-0.1in}
\end{abstract}

\maketitle

\section{Introduction}\label{sec:introduction}

An important question in Riemannian geometry is to investigate the structure of fundamental groups of Riemannian manifolds 
with non-negative sectional curvature. A well known example of this is a theorem of Gromov which states that the fundamental group of a complete Riemannian manifold $M^n$ with non-negative sectional curvature has at most $C(n)$ generators, where $C(n)$ is a constant depending only on the dimension of $M$ (see \cite{Gromov}). 
In addition, the Cheeger-Gromoll splitting theorem, together with a theorem of Wilking, implies that a group $G$ 
is the fundamental group of a non-negatively curved Riemannian manifold if and only if $G$ has a normal subgroup 
isomorphic to ${\mathbb{Z}}^d$  such that the quotient group is finite (see \cite{CheegerGromoll} and \cite[Theorem 2.1]{Wilking2000}).

Under the stronger assumption of positive curvature, the only known further obstructions are the results of Bonnet-Myers
and Synge which together imply that the fundamental group of a positively curved Riemannian manifold is finite 
and, moreover, trivial or ${\mathbb{Z}}_2$ if the dimension of the manifold is even.

As for examples, the largest class of groups which arise as fundamental groups of positively curved manifolds are the spherical space form groups. These are groups that act freely and linearly on spheres (for a complete classification, see \cite[Chapter III]{Wolf}). The first step in the classification of spherical space form groups is to establish that they satisfy the $(p^2)$
and $(2p)$ conditons, which mean respectively that every subgroup of order $p^2$ is cyclic and every involution 
is central. The $(p^2)$ condition was proved by Smith for groups acting freely on a mod $p$ homology sphere
(see \cite{Smith}). Moreover, the $(2p)$ condition holds for groups acting freely on a mod $2$ homology sphere by results
of Milnor and Davis (see \cite{Milnor} and \cite{Davis}).

In 1965, Chern asked if the $(p^2)$ condition holds for the fundamental groups of Riemannian manifolds with positive
sectional curvature. This question was not answered for over 30 years until Shankar proved that there are examples 
for which the $(p^2)$ condition fails for $p=2$ and the $(2p)$ condition fails for all $p$ (see \cite{Shankar}). 
Later, Bazaikin and Grove-Shankar (see \cite{Bazaikin1999} and \cite{GroveShankar}) showed that there are other classes 
of positively curved manifolds which fail to satisfy the $(p^2)$ condition for $p=3$. 
It remains an open problem whether the $(p^2)$ condition holds for $p\geq 5$ (see \cite{GroveShankarZiller}).
 
In the presence of symmetry, much more is known. For example, the only groups that can arise as fundamental groups 
of Riemannian homogeneous spaces with positive sectional curvature are finite subgroups of $SO(3)$ or $SU(2)$
(see \cite{WilkingZiller}). Under more relaxed symmetry assumptions, the most remarkable result in this direction is due to Rong (see \cite{Rong1999}): The fundamental group of an odd-dimensional positively curved Riemannian manifold $M$
with circle symmetry has a cyclic subgroup of index at most a constant $w(n)$ depending only on the dimension of $M$. 
In particular, the $(p^2)$ condition holds for sufficiently large $p$ for this class of manifolds. 
The constant $w(n)$ here is larger than Gromov's Betti number estimate. Part of our motivation is to refine the estimate for $w(n)$ in dimension $13$. Our main result replaces $S^1$ by $T^2$ or $T^3$ and restricts to the class of manifolds 
whose universal covers have the rational cohomology of a Bazaikin space (see \cite{Bazaikin1996} and \cite{FloritZiller}).

\begin{main}\label{thm:index 18}
Let $M^{13}$ be a closed Riemannian manifold with positive sectional curvature.  Suppose that the universal cover of $M$
is a rational cohomology Bazaikin space.
     \begin{enumerate} 
     \item If $M$ admits an effective isometric $T^2$-action, then $\pi_1(M)$ has a cyclic subgroup whose index $D$      
     either is $27$ or divides $18$. Moreover, $D\leq 9$ if the universal cover of $M$ is also a mod $3$ cohomology Bazaikin space.
     \item If $M$ admits an effective isometric $T^3$-action, then $\pi_1(M)$ has a cyclic subgroup of index at most three.
     \end{enumerate}
\end{main}

\begin{nonumberremark}
Note that $\pi_1(M)$ is cyclic in Theorem \ref{thm:index 18} under the stronger assumption of $T^4$ symmetry 
by a result of Frank-Rong-Wang (see \cite{FrankRongWang}). 
\end{nonumberremark}

\begin{nonumberremark}
Theorem \ref{thm:index 18} implies $\pi_1(M)$ satisfies the $(p^2)$ condition for all $p\geq 5$.
Davis has commented to the author that the results in \cite{DavisMilgram} imply that, for all odd primes $p$, 
there exists a closed, simply connected, smooth manifold with the rational cohomology of a Bazaikin space that admits 
a free action by ${\mathbb{Z}}_p\times{\mathbb{Z}}_p$. Therefore, one cannot prove Theorem \ref{thm:index 18}
after dropping both the curvature and symmetry assumptions.
\end{nonumberremark}

We now discuss a corollary to Theorem \ref{thm:index 18}. 
The only simply connected, closed $13$-dimensional manifolds known to admit positive curvature are ${\mathbb{S}}^{13}$ 
and the Bazaikin spaces (see \cite{Ziller2007}). By a result of Kennard (see \cite[Corollary 6.3]{Kennard}), 
if $M^{13}$ is a closed, positively curved Riemannian manifold with $T^2$ symmetry whose universal cover is a rational sphere, then $\pi_1(M)$ is cyclic. Combining this result with Theorem \ref{thm:index 18}, we get the following corollary:

\begin{nonumbercorollary}
If $M^{13}$ is a closed, positively curved Riemannian manifold with $T^2$ symmetry, and if the universal cover 
of $M$ has the cohomology of one of the known positively curved $13$-dimensional examples, then $\pi_1(M)$ 
has a cyclic subgroup of index $1, 2, 3, 6$, or $9$.
\end{nonumbercorollary}

We now discuss the tools used in the proof. In addition to the well known results such as Berger's fixed point theorem 
and Wilking's connectedness lemma (see Section \ref{sec:preliminaries}), the main new tool is Lemma \ref{lem:Z_p}.
This is a structural result for groups acting freely on positively curved manifolds with circle symmetry that generalizes an obstruction from Kennard (see \cite[Proposition 5.1]{Kennard}). Together with a result of Davis and Weinberger 
(see Theorem \ref{thm:Davis}), Lemma \ref{lem:Z_p} places strong restrictions on the Sylow subgroups of the fundamental group. 
In fact, we show that after possibly passing to a subgroup $B$ of index $2, 3, 6$, or $9$, every Sylow subgroup is cyclic. 
Burnside's classification (see Section \ref{sec:preliminaries}), together with Lemma \ref{lem:Z_p}, then implies that $B$
itself has a cyclic subgroup of index $1$ or $3$. In addition, results from equivariant cohomology 
are applied to calculate the fixed point components of the circle action. The key here is that we fix the rational type 
of the universal cover. Finally, to analyze the case in which the manifold is a mod $3$ cohomology Bazaikin space, 
we modify an argument due to Heller to further restrict the Sylow $3$-subgroup of the fundamental group 
(see Section \ref{sec:Z_3}).

This article is organized as follows. Section \ref{sec:preliminaries} provides basic results which will be used 
throughout the paper. Section \ref{sec:Bazaikin} states the definition of Bazaikin spaces as well as a lemma about groups 
acting freely and isometrically on a rational cohomology Bazaikin space. In Section \ref{sec:obstruction}, we prove Lemma \ref{lem:Z_p}. Theorem \ref{thm:index 18} in the case of rational cohomology Bazaikin spaces is proved in Section \ref{sec:Q}. 
Finally, in Section \ref{sec:Z_3}, we complete the proof of Theorem \ref{thm:index 18} by considering the case of mod $3$ cohomolgy Bazaikin spaces. 
\vspace{0.3cm}\\
\noindent{\textbf{Acknowledgements}}
This paper is part of the author's Ph.D. thesis. The author would like to thank her advisor, Lee Kennard, for his support
and helpful comments and suggestions.

\section{Preliminaries}\label{sec:preliminaries}

This section consists of three parts. The first part states some results about positively curved manifolds. In the second part,
we provide a theorem from equivariant cohomology. The last part discusses some tools from group theory.

One of the most powerful results in the theory of positively curved manifolds is the following theorem due to Wilking 
(see \cite[Theorem 2.1]{Wilking2003}):

\begin{theorem}[Connectedness lemma]\label{thm:connectedness lemma}
Let $M^n$ be a closed positively curved Riemannian manifold. If $N^{n-k}$ is a closed totally geodesic submanifold of $M$, 
then the inclusion $N^{n-k}\embedded M^n$ is $(n-2k+1)$-connected. Moreover, if $N^{n-k}$ is fixed pointwise 
by the action of a Lie group $G$ which acts isometrically on $M$, then the inclusion is $(n-2k+1+\delta(G))$-connected, 
where $\delta(G)$ is the dimension of the principal orbit.
\end{theorem}

The following is a result of Poincar\'e duality and refines the conclusion of Theorem \ref{thm:connectedness lemma}:

\begin{lemma}[{\cite[Lemma 2.2]{Wilking2003}}]\label{lem:periodicity}
Let $M^n$ be a closed orientable smooth manifold and let $N^{n-k}$ be a closed orientable submanifold.
If the inclusion $N^{n-k}\embedded M^n$ is $(n-k-l)$-connected and $n-k-2l>0$, then there exists 
$e\in H^k(M;\mathbb{Z})$ such that the map $\cup e:H^i(M;\mathbb{Z})\to H^{i+k}(M;\mathbb{Z})$
is surjective for $l\leq i< n-k-l$ and injective for $l<i\leq n-k-l$.
\end{lemma}

The next result is due to Frank, Rong, and Wang (see \cite[Corollaries 1.7 and 1.9]{FrankRongWang}).

\begin{corollary}\label{cor:codimension 2}
Let $M^n$ be a closed odd-dimensional Riemannian manifold with positive sectional curvature. If $n\geq 5$ and $M$ has a closed 
totally geodesic submanifold $N$ of codimension two, then the universal covering spaces of $M$ and $N$ are homotopy spheres,
and $\pi_1(M)\cong\pi_1(N)$ is cyclic.
\end{corollary}

For our purposes, we need the following generalization of Corollary \ref{cor:codimension 2}:

\begin{proposition}\label{prop:codimension 2}
Let $M^n$ be a closed odd-dimensional Riemannian manifold with positive sectional curvature. Suppose that $n\geq 5$ 
and $M$ has a closed totally geodesic submanifold $N$ of codimension two. If $\Gamma$ is a finite group that acts freely
on both $M$ and $N$, then $\Gamma$ is cyclic. 
\end{proposition}

\begin{proof}
Consider the Riemannian covering map $q:M\to {M/{\Gamma}}$. By \cite[Proposition 1.40]{Hatcher}, we have
$\Gamma\cong{\pi_1(M/{\Gamma})}/{q_*(\pi_1(M))}$. In addition, $M/{\Gamma}$ and $N/{\Gamma}$ 
are closed odd-dimensional positively curved manifolds and $N/{\Gamma}$ is a totally geodesic submanifold 
of $M/{\Gamma}$ of codimension two. Hence $\pi_1(M/{\Gamma})$ is cyclic by Corollary \ref{cor:codimension 2}. 
This implies that $\Gamma\cong{\pi_1(M/{\Gamma})}/{q_*(\pi_1(M))}$ is cyclic.
\end{proof}

We end the first part of this section with a generalization of Berger's theorem about torus actions on positively curved Riemannian manifolds of even dimension (see \cite{Berger}). The statement in odd dimensions is due to Sugahara \cite{Sugahara}
(cf. \cite{GroveSearle}).

\begin{theorem}\label{thm:Berger}
Let $M$ be a closed, odd-dimensional Riemannian manifold with positive sectional curvature. 
If $M$ admits an effective isometric $T^k$-action, then there is a circle orbit. In particular, there exists $T^{k-1}\subseteq T^k$ with non-empty fixed point set.
\end{theorem}

One of the main tools used in the proof of Theorem \ref{thm:index 18} is the relationship between the cohomology 
of a manifold $M$ and that of the fixed point set $M^{S^1}$ of a circle acting on $M$. 
For our purposes, we need the following result. It is proved by applying tools from equivariant cohomology 
(see \cite[Theorem 3.8.12 and Theorem 3.10.4]{AlldayPuppe}):

\begin{theorem}\label{thm:Allday}
If $M$ is a compact manifold which admits a smooth $S^1$-action, then the rational Betti numbers satisfy
$$\sum_i b_i(M^{S^1};\mathbb{Q})\leq\sum_i b_i(M;\mathbb{Q}).$$
Moreover, if $\sum_i b_i(M^{S^1};\mathbb{Q})=\sum_i b_i(M;\mathbb{Q})$ and if $H^*(M;\mathbb{Q})$ has $r$ generators 
of even degree and $s$ generators of odd degree, then for any component $F$ of $M^{S^1}$, $H^*(F;\mathbb{Q})$ 
has at most $r$ generators of even degree and at most $s$ generators of odd degree.
\end{theorem}

We end this section with some results from the theory of finite groups and free actions by finite groups.
The first result is due to Davis and Weinberger (see \cite[Theorem D]{Davis}).

\begin{theorem}\label{thm:Davis}
Let $M^{4k+1}$ be a closed manifold such that $\sum_{i=0}^{2k}(-1)^i\dim H^i(M;\mathbb{Q})$ is odd. 
If $G$ is a finite group that acts freely on $M$ such that the induced action on $H^*(M;\Q)$ is trivial, 
then $G$ is the direct product of a cyclic $2$-group and a group $\Gamma$ of odd order. 
\end{theorem}

\begin{remark}
For $G$ and $\Gamma$ as in Theorem \ref{thm:Davis}, $G$ is cyclic if and only if $\Gamma$ is cyclic, and, more generally,
$G$ has a cyclic subgroup of index $r$ if and only if $\Gamma$ has a cyclic subgroup of index $r$.
\end{remark}

In the proof of Theorem \ref{thm:index 18}, we are interested in the Sylow $p$-subgroups of $\pi_1(M)$.
Theorem \ref{thm:Davis} states a condition under which the Sylow $2$-subgroup of a finite group acting freely on 
a $(4k+1)$-dimensional manifold is cyclic. Theorem \ref{thm:Wolf} provides a condition under which Sylow
$p$-subgroups for odd $p$ are cyclic.

\begin{theorem}[{\cite[Theorem 5.3.2]{Wolf}}]\label{thm:Wolf}
If $\Gamma$ is a finite group of odd order, then the following statements are equivalent:
       \begin{enumerate}
       \item $\Gamma$ satisfies every $(p^2)$ condition, i.e. $\Gamma$ does not contain a copy 
       of ${\mathbb{Z}}_p\times{\mathbb{Z}}_p$.
       \item Every Sylow $p$-subgroup of $\Gamma$ is cyclic. 
       \end{enumerate}
\end{theorem}

Odd-order groups which satisfy all $(p^2)$ conditions have a nice presentation and enjoy some properties 
which will be discussed in what follows (see \cite[Theorem 5.4.1]{Wolf}):

\begin{theorem}[Burnside]\label{thm:Burnside}
If $G$ is a finite group in which every Sylow subgroup is cyclic, then $G$ is generated by two elements 
$A$ and $B$ with defining relations
$$A^m=B^n=1,\hspace{0.5cm}BAB^{-1}=A^r;$$
$$((r-1)n,m)=1,\hspace{0.5cm}r^n\cong 1({\rm{mod}}~m).$$
\end{theorem}

\begin{definition}
The collection of all groups of the form 
$$\langle A,B: A^m=B^n=1, BAB^{-1}=A^r\rangle~\text{where}~((r-1)n,m)=1~\text{and}~r^n\cong 1(\text{mod}~m)$$ 
will be denoted by $\mathcal{C}$. We partition the collection $\mathcal{C}$ into groups ${\mathcal{C}}_d$, where $d$ denotes the order of $r$ in the multiplicative group of units modulo $m$. 
\end{definition}

\begin{remark}\label{remark:C_d}
Note that every $\Gamma\in{\mathcal{C}}_d$ has a normal cyclic subgroup of index $d$. Indeed, the subgroup $H$
generated by $A$ and $B^d$ is a normal cyclic subgroup of index $d$ in $\Gamma$ (see \cite[Theorem 5.5.1]{Wolf}).
It can be proved moreover that $H$ is not strictly contained in any cyclic subgroup.
\end{remark}

The last collection of algebraic tools which we require are some basic results about $p$-groups and normal $p$-complements. 
Let $p$ be a prime. It is a well known fact that every $p$-group $P$ with $|P|=p^m$ has a normal subgroup of order $p^i$ 
for all $1\leq i\leq m$. Moreover, the classification of groups of order $p^3$ (see \cite[p. 140]{Burnside}) implies that any group 
of order $27$ is isomorphic to ${\mathbb{Z}}_{27}$, ${\mathbb{Z}}_9\times{\mathbb{Z}}_3$, ${\mathbb{Z}}_3\times{\mathbb{Z}}_3\times{\mathbb{Z}}_3$, ${\mathbb{Z}}_9\rtimes{\mathbb{Z}}_3$, or 
\begin{equation*}
U(3,3):=\left\{
\begin{pmatrix}
1 & x & z\\
0 & 1 & y\\
0 & 0 & 1
\end{pmatrix}
: x,y,z\in{\mathbb{Z}}_3\right\}.
\end{equation*}
We also require a result about groups of order $81$. Every non-cyclic group of order $81$ contains either a copy 
of ${\mathbb{Z}}_9\times{\mathbb{Z}}_3$ or a copy of ${\mathbb{Z}}_3\times{\mathbb{Z}}_3\times{\mathbb{Z}}_3$
(see \cite[pp. 140 and 145]{Burnside}). These facts imply the following proposition: 

\begin{proposition}\label{prop:3-groups}
If $G$ is a $3$-group which contains ${\mathbb{Z}}_3\times{\mathbb{Z}}_3$ but does not contain ${\mathbb{Z}}_9\times{\mathbb{Z}}_3$, ${\mathbb{Z}}_3\times{\mathbb{Z}}_3\times{\mathbb{Z}}_3$, 
or $U(3,3)$, then $G$ is isomorphic to ${\mathbb{Z}}_3\times{\mathbb{Z}}_3$ or ${\mathbb{Z}}_9\rtimes{\mathbb{Z}}_3$.
\end{proposition}

\begin{proof}
By the discussion above, we only need to prove that the order of $G$ is at most $81$. Suppose by way of contradiction 
that the order of $G$ is bigger than $81$ and let $P_{81}$ be a normal subgroup of $G$ of order $81$. 
If $P_{81}$ is non-cyclic, then it contains a copy of  ${\mathbb{Z}}_9\times{\mathbb{Z}}_3$
or ${\mathbb{Z}}_3\times{\mathbb{Z}}_3\times{\mathbb{Z}}_3$, a contradiction. Hence we may assume
that $P_{81}=\langle g\rangle$ is cyclic. Let $L$ denote the maximal cyclic subgroup of $G$ which contains $P_{81}$. 
Since $L$ is cyclic and $G$ contains a copy of ${\mathbb{Z}}_3\times{\mathbb{Z}}_3$, there exists 
${\mathbb{Z}}_3=\langle h\rangle\subseteq G$ such that $\langle h\rangle\cap L=\{1\}$. 
Since $\langle h\rangle$ normalizes $\langle g^3\rangle$, we can form the subgroup $K=\langle g^3\rangle\langle h\rangle$ which is a group of order $81$. Note that $K$ is non-cyclic.  Indeed, if $K$ is cyclic, then it would be abelian.
This implies that every element in $K$ is of the form $(g^3)^ih^j$ and has order at most $27$. 
As a non-cyclic group of order $81$, $K$ contains either a copy of ${\mathbb{Z}}_9\times{\mathbb{Z}}_3$ 
or a copy of ${\mathbb{Z}}_3\times{\mathbb{Z}}_3\times{\mathbb{Z}}_3$. This is a contradiction.
\end{proof}

The normal rank of a $p$-group $P$ is the largest integer $k$ such that $P$ contains an elementary abelian normal subgroup 
of order $p^k$. Our final algebraic result is the following:

\begin{theorem}[{\cite[p. 257]{Gorenstein}}]\label{thm:normal complement}
Let $G$ be a finite group and let $p$ be the smallest prime dividing the order of $G$. Let $P$ denote the Sylow $p$-subgroup 
of $G$. Suppose that $P$ is cyclic if $p=2$ and that the normal rank of $P$ is at most two otherwise. Then there exists 
a normal $p$-complement of $G$, i.e., a normal subgroup $N$ of $G$ such that $G=PN$ and $P\cap N=\{1\}$. 
\end{theorem}

\section{Bazaikin Spaces}\label{sec:Bazaikin}

Besides spherical space forms, the only $13$-dimensional manifolds known to admit positive curvature are a family of biquotients called Bazaikin spaces.

Biquotients are defined in the following way. Let $G$ be a compact Lie group and let $U$ be a subgroup of $G\times G$. 
There exists an action of $U$ on $G$ defined by $(u_1,u_2)\centerdot g=u_1gu_2^{-1}$. In case the action is free, 
the quotient space is called a biquotient and is denoted by $G/\hspace{-0.1cm}/U$.

Bazaikin spaces are examples of biquotients but here we give a slightly different description (for more details, see \cite{FloritZiller}). Let $q=(q_1,\ldots,q_5)$ be a five tuple of integers and let $q_0:=\sum q_i$. There exists an injective homomorphism
\begin{align*}
Sp(2)\times S^1 &\longrightarrow U(5)\times U(5)\\
(A,z) &\longmapsto (\diag(z^{q_1},\ldots z^{q_5}),\diag(z^{q_0},A)),
\end{align*}
where we consider $Sp(2)$ as a subgroup of $SU(4)$ via the inclusion
\begin{equation*}
A+Bj\in Sp(2)\longmapsto
\begin{pmatrix}
A & B\\
-\bar{B} & \bar{A}
\end{pmatrix}. 
\end{equation*}
The above homomorphism gives an action of $Sp(2)\times S^1$ on $U(5)$ defined by 
$(A,z)\centerdot g=\diag(z^{q_1},\ldots z^{q_5})g\diag(z^{-q_0},A^{-1})$ which restricts to an action of $Sp(2)\times S^1$ 
on $SU(5)$. The kernel of this action is ${\mathbb{Z}}_2$ and hence we obtain an effective action 
of $Sp(2).S^1:=(Sp(2)\times S^1)/{{\mathbb{Z}}_2}$ on $SU(5)$. The action of $Sp(2).S^1$ on $SU(5)$ is free 
if and only if all the $q_i$ are odd and $\gcd(q_{\sigma(1)}+q_{\sigma(2)},q_{\sigma(3)}+q_{\sigma(4)})=2$ 
for all permutations $\sigma\in S_5$. In this case, the quotient space $B_q={SU(5)}/{Sp(2).S^1}$
is called a Bazaikin space. The Bazaikin space $B_q$ admits positive sectional curvature 
if and only if $q_i+q_j>0$ (or $<0$) for all $i<j$.

\begin{proposition}[{\cite{Bazaikin1996}}]\label{prop:cohomology of Bazaikin spaces}
The integral cohomology groups of the Bazaikin space $B_q$ are given by:
\begin{equation*}
H^k(B_q;\mathbb{Z})=
\begin{cases}
\mathbb{Z} & \text{if}~k=0,2,4,9,11,13\\
{\mathbb{Z}}_m & \text{if}~k=6,8\\
0 & otherwise
\end{cases}
\end{equation*}
Here, $m=\frac{1}{8}\sum_{i<j<k}q_iq_jq_k$.
Moreover, if $x$ is a generator of $H^2(B_q;\mathbb{Z})$, then $x^k$ is a generator of $H^{2k}(B_q;\mathbb{Z})$ 
for $k=2,3,4$. In particular, $B_q$ has the rational cohomology of $\mathbb{CP}^2\times{\mathbb{S}}^9$ 
and the mod $3$ cohomology of either ${\mathbb{CP}}^2\times\s^9$ or ${\mathbb{CP}}^4\times\s^5$.
\end{proposition}

One of the crucial steps in the proof of Theorem \ref{thm:index 18}, is to find a subgroup of $\pi_1(M)$ of minimal index 
which satisfies the conditions of Theorem \ref{thm:Davis}. The following lemma provides such a subgroup:

\begin{lemma}\label{lem:trivial action on cohomology}
Let $M^{13}$ be a closed Riemannian manifold with positive sectional curvature whose universal cover  
has the rational cohomology of a Bazaikin space. Let $\Gamma:=\pi_1(M)$. Then $\Gamma$ has a subgroup $\Gamma_1$ 
of index at most two that acts trivially on $H^*(\widetilde{M};\mathbb{Q})$.
\end{lemma}

\begin{proof}
Since $\Gamma$ acts isometrically and freely on $\widetilde{M}$, Weinstein's theorem implies that $\Gamma$ acts 
by orientation-preserving homeomorphisms on $\widetilde{M}$ and hence trivially on $H^{13}(\widetilde{M};\mathbb{Q})$.
Now, $H^*(\widetilde{M};\mathbb{Q})\cong H^*({\mathbb{CP}}^2\times{\mathbb{S}}^9;\mathbb{Q})$ by Proposition \ref{prop:cohomology of Bazaikin spaces}, so $H^*(\widetilde{M};\mathbb{Q})$ is generated by some 
$\alpha\in H^2(\widetilde{M};\mathbb{Q})$ and $\beta\in H^9(\widetilde{M};\mathbb{Q})$. Since $\Gamma$
acts by homomorphisms on $H^*(\widetilde{M};\mathbb{Q})$, there is a subgroup $\Gamma_1\trianglelefteq\Gamma$ 
of index at most two fixing $\alpha$. Indeed, $\Gamma_1$ is the kernel of the homomorphism $f:\Gamma\to\Z_2=\{\pm 1\}$
defined by $f(\gamma)={\epsilon}_{\gamma}$, where $\gamma^*(\alpha)={\epsilon}_{\gamma}\alpha$. 
Since $\Gamma_1$ also fixes $\alpha^2$ and $\alpha^2\beta\in H^{13}(\widetilde{M};\mathbb{Q})$, 
it acts trivially on $H^*(\widetilde{M};\mathbb{Q})$.
\end{proof}

\section{A new obstruction}\label{sec:obstruction}

One of the key tools in the proof of Theorem \ref{thm:index 18} has to do with the structure of finite groups which act freely 
and by isometries on a positively curved manifold with circle symmetry. Here, we generalize an obstruction from Kennard \cite[Proposition 5.1]{Kennard} (cf. SunWang \cite[Lemma 1.5]{SunWang}) to restrict the structure of such groups.

\begin{lemma}\label{lem:Z_p}
Let $M$ be a closed positively curved Riemannian manifold. Let $\Gamma$ be a group of odd order which acts freely 
and by isometries on $M$. Assume $M$ admits an effective isometric $S^1$-action which commutes with the action of $\Gamma$. 
If $H=\langle\alpha\rangle$ is a normal cyclic subgroup of $\Gamma$ that is not strictly contained in a larger cyclic subgroup, 
then $|{\Gamma}/H|$ divides the Lefschetz number
${\text{Lef}}(\alpha;\bar{M}, M^{S^1})=\sum_i(-1)^i\tr(\alpha^*:H^i(\bar{M}, M^{S^1};\Q)\to H^i(\bar{M}, M^{S^1};\Q))$ 
of $\alpha$, where $\bar{M}=M/{S^1}$.
\end{lemma}

Note that when applied to $\Gamma=\Z_p\times\Z_p$, Lemma \ref{lem:Z_p} implies that $p$ divides 
${\text{Lef}}(\alpha;\bar{M}, M^{S^1})$. Recall also from Remark \ref{remark:C_d} that for $\Gamma\in{\mathcal{C}}_d$, 
the subgroup $H$ generated by $A$ and $B^d$ is a normal cyclic subgroup of index $d$ in $\Gamma$ 
that is not strictly contained in any cyclic subgroup. Hence Lemma \ref{lem:Z_p} applies to $\Gamma\in{\mathcal{C}}_d$ 
and shows that $d$ divides ${\text{Lef}}(\alpha;\bar{M}, M^{S^1})$. 

\begin{proof}
Since the action of $\alpha$ on $M$ commutes with the circle action, 
$\alpha$ acts on $\bar{M}$. Consider the fixed point set ${\bar{M}}^{\alpha}$. 
We claim that $\chi({\bar{M}}^{\alpha})={\text{Lef}}(\alpha;\bar{M}, M^{S^1})$. 
In order to prove this, note that since the action of $\Gamma$ on $M$ commutes with the $S^1$-action, 
$\alpha$ acts on $M^{S^1}$. Moreover, the action of $\alpha$ on $M^{S^1}$ is free since $\Gamma$ acts freely on $M$. 
Therefore, $({M^{S^1}})^{\alpha}=\emptyset$ 
and $\chi({\bar{M}}^{\alpha})=\chi({\bar{M}}^{\alpha},({M^{S^1}})^{\alpha})$, where 
$\chi({\bar{M}}^{\alpha},({M^{S^1}})^{\alpha})=\sum_i(-1)^i\dim H^i({\bar{M}}^{\alpha},({M^{S^1}})^{\alpha};\Q)$.
The claim follows since $\chi({\bar{M}}^{\alpha},({M^{S^1}})^{\alpha})={\text{Lef}}(\alpha;\bar{M}, M^{S^1})$
by \cite[Exercise 2.C.4]{Hatcher} (cf. \cite[p. 250]{AlldayPuppe}). 
\par
We may assume that ${\bar{M}}^{\alpha}$ is non-empty since otherwise ${\text{Lef}}(\alpha;\bar{M}, M^{S^1})$ 
is zero by the claim and we are done. Since $H=\langle\alpha\rangle$ is a normal subgroup of $\Gamma$, 
$\Gamma$ acts on ${\bar{M}}^{\alpha}$. The kernel of this action is $H$ and hence it induces an action 
of $G={\Gamma}/H$ on ${\bar{M}}^{\alpha}$. 
Note that $G$ acts on the set of components of ${\bar{M}}^{\alpha}$ and hence partitions the set of components into orbits. 
Let $\{F_i\}_{i=1}^m$ be the set of connected components of ${\bar{M}}^{\alpha}$ and let $G_{F_i}$ 
denote the isotropy group of $F_i$. Either ${G_{F_i}}=1$ for all $i$ or ${G_{F_i}}\neq 1$ for some $i$. 
\par
Suppose first that ${G_{F_i}}=1$ for all $i$. In this case, we have $|G.F_i|=|G|$ for all $i$. 
This means that each orbit consists of $|G|$ components. In addition, any two components in the same orbit are homeomorphic and hence have the same Euler characteristic. Therefore, $|{\Gamma}/H|=|G|$ divides $\chi({\bar{M}}^{\alpha})$ 
and hence ${\text{Lef}}(\alpha;\bar{M}, M^{S^1})$.
\par
Now, consider the case in which ${G_{F_i}}\neq 1$ for some $i$. In this case, there exists non-trivial $[g]\in G$ 
such that $[g].F_i=F_i$ and hence $g.F_i=F_i$. Let $\pi:M\rightarrow{\bar{M}}$ denote the quotient map.
Note that $g$ acts isometrically on $\pi^{-1}(F_i)$. Since this action commutes with the circle action on $\pi^{-1}(F_i)$, 
$g$ preserves some circle orbit $C$ in $\pi^{-1}(F_i)$ (see \cite[Theorem A]{Rong2005}). But $\alpha$ also acts on $C$ 
and hence $\langle g,\alpha\rangle$ acts on $C$. This implies that $\langle g,\alpha\rangle$ is cyclic. 
Since $g\notin\langle\alpha\rangle$, we conclude that $\langle\alpha\rangle$ is strictly contained in $\langle g,\alpha\rangle$,
a contradiction.
\end{proof}

\begin{remark}\label{remark:Z_p trivial}
In our applications, except the first case in the proof of Lemma \ref{lem:dim 5}, the cohomology groups 
$H^i(\bar{M},M^{S^1};\Q)$ have dimension at most one.  Hence the induced action of $\alpha$ 
on $H^*(\bar{M},M^{S^1};\Q)$ is trivial and we may replace ${\text{Lef}}(\alpha;\bar{M}, M^{S^1})$ 
by $\chi(\bar{M},M^{S^1})$.  
\end{remark}

When applying Lemma \ref{lem:Z_p}, we need to figure out the cohomology groups $H^i(M/{S^1},M^{S^1})$.
The main tool in calculating these groups is the Smith-Gysin sequence which relates the relative cohomology groups  $H^i(M/{S^1},M^{S^1})$ to the cohomology groups of $M$ and $M^{S^1}$.

\begin{theorem}[Smith-Gysin sequence, {\cite[p. 161]{Bredon}}]
If $S^1$ act on a paracompact space $X$, then there exists a long exact sequence
\begin{align*}
\ldots &\longrightarrow H^i(X/{S^1},X^{S^1})\longrightarrow H^i(X)\longrightarrow H^{i-1}(X/{S^1},X^{S^1})\oplus H^i(X^{S^1})\\
 & \longrightarrow H^{i+1}(X/{S^1},X^{S^1})\longrightarrow\ldots
\end{align*}
called the Smith-Gysin sequence. Here, we take coefficients in $\mathbb{Q}$.
\end{theorem}

\section{Proof of theorem \ref{thm:index 18} for rational cohomology Bazaikin spaces}\label{sec:Q}

In this section, we prove Theorem \ref{thm:index 18} for rational cohomology Bazaikin spaces.
We equip the universal cover of $M$ with the pull back metric. We also lift the torus action to the universal cover of $M$
(see \cite [Theorem I.9.1]{Bredon}) and then break the proof into subsections. Section \ref{sec:fixed point} discusses the case 
in which the lifted torus action on $\widetilde{M}$ has non-empty fixed point set. In Section \ref{sec:1 or 3}, 
we consider the case in which there is a circle inside $T^2$ whose action on $\widetilde{M}$ has a fixed point component 
of dimension one or three. In Section \ref{sec:5}, we consider the case of a circle inside $T^2$ with five-dimensional fixed point set. Finally, in Section \ref{sec:conclusion}, we conclude the proof. Before proceeding to the proof, we prove the following lemma:

\begin{lemma}\label{lem:totally geodesic}
Let $M^{13}$ be a closed, positively curved Riemannian manifold. If the universal cover of $M$ is a rational cohomology Bazaikin space, then $M$ does not have any totally geodesic submanifolds of codimension two or four.
\end{lemma}

\begin{proof}
We proceed by contradiction. Let $N$ be a totally geodesic submanifold of $M$. If the codimension of $N$ equals two, 
then by Corollary \ref{cor:codimension 2}, $\widetilde{N}$ and $\widetilde{M}$ are homotopy spheres and this contradicts 
the assumption that $\widetilde{M}$ is a rational cohomology Bazaikin space. 
\par
Suppose now that the codimension of $N$ equals four. Theorem \ref {thm:connectedness lemma} 
implies that the inclusion $N\embedded M$ is $6$-connected. Therefore, by Lemma \ref{lem:periodicity}, 
the homomorphism $\cup e:H^i(M;\Z)\to H^{i+4}(M;\Z)$ is surjective for $3\leq i<6$ and injective for $3<i\leq 6$. 
This implies that $H^5(M;\Z)\cong H^9(M;\Z)$. Recall that $H^*(M;\Q)\cong H^*(\widetilde{M};\Q)^{\Gamma}$, 
where $\Gamma=\pi_1(M)$ and where $H^*(\widetilde{M};\Q)^{\Gamma}$ denotes the subring of elements invariant 
under the induced action of $\Gamma$ (see \cite[Theorem III.2.4]{Bredon}). Since $H^5(\widetilde{M};\Q)=0$, 
it follows that $H^5(M;\Q)=0$. On the other hand, $H^9(\widetilde{M};\Q)\cong\Q$ and $\Gamma$ acts trivially 
on $H^9(\widetilde{M};\Q)$ by the proof of Lemma \ref{lem:trivial action on cohomology}. Hence $H^9(M;\Q)\cong\Q$
and we have a contradiction.
\end{proof}

\begin{notation}
Throughout the rest of paper ${\widetilde{M}}^G$ (resp. $M^G$), where $G=S^1$ or $T^2$, denotes the fixed point set 
of the action of $G$ on ${\widetilde{M}}$ (resp. $M$). Similarly, ${\widetilde{M}}_x^G$ (resp. $M_x^G$) 
denotes the component of ${\widetilde{M}}^G$ (resp. $M^G$) containing $x$.
\end{notation}

\subsection{Torus action with fixed point}\label{sec:fixed point}

The first case in the proof of Theorem \ref{thm:index 18} is when the lifted torus action on $\widetilde{M}$ has a fixed point. 
In this case, we get a better bound for the index of cyclic subgroups of minimal index.

\begin{lemma}\label{lem:$T^2$-action with fixed point}
Let $M^{13}$ be a closed Riemannian manifold with positive sectional curvature which admits an effective isometric 
$T^2$-action. Suppose that the universal cover of $M$ is a rational cohomology Bazaikin space. If the lifted torus action 
on $\widetilde{M}$ has a fixed point, then $\pi_1(M)$ has a cyclic subgroup of index at most three. 
\end{lemma}

Note that Theorem \ref{thm:index 18} in the case of $T^3$ symmetry follows immediately since some $T^2\subseteq T^3$
has a fixed point by Theorem \ref{thm:Berger}.

\begin{proof}
Let $x$ be a fixed point for the $T^2$-action on $\widetilde{M}$ 
and let $\codim(P\subseteq Q)$ denote the codimension of $P$ in $Q$. Borel's formula (see \cite[Theorem 5.3.11]{AlldayPuppe}) states that
$$\sum_{S^1\subseteq T^2}\codim({\widetilde{M}}_x^{T^2}\subseteq {\widetilde{M}}_x^{S^1})=\codim({\widetilde{M}}_x^{T^2}\subseteq\widetilde{M}),$$
where we take the sum over all the circles inside $T^2$. Note that only finitely many terms contribute to this sum.
Note also that any fixed point component of an effective torus action on a positively curved manifold has even codimension.
We break the proof into cases:
      \begin{itemize}
      \item $\dim({\widetilde{M}}_x^{T^2})=1$. In this case, ${\widetilde{M}}_x^{T^2}$ is diffeomorphic to $\s^1$. 
       Set $\Gamma=\pi_1(M)$ and let $\Gamma_2\subseteq\Gamma$ be the subgroup of elements mapping the component   
       ${\widetilde{M}}_x^{T^2}$ to itself. Note that $\Gamma$ acts on ${\widetilde{M}}^{T^2}$ because its action 
       on $\widetilde{M}$ commutes with the $T^2$-action. Since ${\widetilde{M}}_x^{T^2}$ is a circle, $\Gamma_2$ is cyclic. 
       In order to calculate the index of $\Gamma_2$, note that $\Gamma$ acts on the set of components 
       of ${\widetilde{M}}^{T^2}$ and hence
       $$[\Gamma:\Gamma_2]\leq\#\{\text{components~of~}{\widetilde{M}}^{T^2}\}.$$
       In addition (for the second inequality, see \cite[Corollary 3.1.14]{AlldayPuppe}), 
       $$\#\{\text{components~of~}{\widetilde{M}}^{T^2}\}\leq\sum_i\dim H^{2i}({\widetilde{M}}^{T^2};       
        \mathbb{Q})\leq\sum_i\dim H^{2i}(\widetilde{M};\mathbb{Q})=3.$$
       Therefore, $[\Gamma:\Gamma_2]\leq 3$, as required.
      \item $\dim({\widetilde{M}}_x^{T^2})\geq 3$ and there exists $S^1\subseteq T^2$ 
       with $\codim({\widetilde{M}}_x^{T^2}\subseteq{\widetilde{M}}_x^{S^1})=2$. 
       Let $\Gamma_2$ be the subgroup of $\Gamma$ which acts on ${\widetilde{M}}_x^{T^2}$. 
       As argued in the previous case, the index of $\Gamma_2$ in $\Gamma$ is at most three.
       Since $\Gamma_2$ acts on ${\widetilde{M}}_x^{T^2}$, it also acts on ${\widetilde{M}}_x^{S^1}$. 
       Proposition \ref{prop:codimension 2} now implies that $\Gamma_2$ is cyclic.
      \item $\dim({\widetilde{M}}_x^{T^2})\geq 3$ 
       and $\codim({\widetilde{M}}_x^{T^2}\subseteq{\widetilde{M}}_x^{S^1})\neq 2$ for all $S^1\subseteq T^2$.
       By Lemma \ref{lem:totally geodesic}, the dimension of ${\widetilde{M}}_x^{S^1}$ is at most seven. Therefore,
       $$7\leq\dim({\widetilde{M}}_x^{T^2})+4\leq\dim({\widetilde{M}}_x^{T^2})+\codim({\widetilde{M}}_x^{T^2}\subseteq{\widetilde{M}}_x^{S^1})=\dim({\widetilde{M}}_x^{S^1})\leq    
       7$$
       for all $S^1\subseteq T^2$ with $\codim({\widetilde{M}}_x^{T^2}\subseteq{\widetilde{M}}_x^{S^1})>0$.
       Hence equality holds and we have $\dim({\widetilde{M}}_x^{T^2})=3$ and $\dim({\widetilde{M}}_x^{S^1})=7$.       
       On one hand, the right-hand side of Borel's formula equals 10. On the other hand, the left-hand side is a multiple of four.
       This is a contradiction.\qedhere
     \end{itemize}
\end{proof}

\subsection{Fixed point component of dimension one or three}\label{sec:1 or 3}

In the presence of $T^2$ symmetry, Theorem \ref{thm:Berger} guarantees existence of a circle whose fixed point set
is non-empty. In this section, we prove Theorem \ref{thm:index 18} in the case where this fixed point set has a component 
of dimension one or three. 

\begin{lemma}\label{lem:dimension 1 and 3}
Let $M^{13}$ and $T^2$ be as in Theorem \ref{thm:index 18}. If there exists $S^1\subseteq T^2$ 
such that ${\widetilde{M}}^{S^1}$ has a component ${\widetilde{M}}_x^{S^1}$ of dimension one or three,
then $\Gamma:=\pi_1(M)$ has a cyclic subgroup of index dividing six. 
\end{lemma}

\begin{proof}
Let $\Gamma_1\subseteq\Gamma$ be the subgroup that acts trivially on $H^*(\widetilde{M};\Q)$. Recall by the proof of Lemma
\ref{lem:trivial action on cohomology} that $\Gamma_1$ has index at most two. By Theorem \ref{thm:Davis}, we have $\Gamma_1\cong{\mathbb{Z}}_{2^a}\times\Gamma'_1$ for some $a\geq 0$ and some group $\Gamma'_1$ of odd order.
Next, let  $\Gamma'_2\subseteq\Gamma'_1$ be the subgroup preserving the component ${\widetilde{M}}_x^{S^1}$.
As in the proof of Lemma \ref{lem:$T^2$-action with fixed point}, we see that ${\widetilde{M}}^{S^1}$
has at most three components and therefore that $\Gamma'_2\subseteq\Gamma'_1$ has index at most three. 
Note in addition that the index is a divisor of three because $\Gamma'_1$ has odd order. We claim that $\Gamma'_2$
is cyclic. In order to prove this, we consider each case separately.
      \begin{itemize}
      \item $\dim({\widetilde{M}}_x^{S^1})=1$. Since ${\widetilde{M}}_x^{S^1}$ is diffeomorphic to $\s^1$,
      $\Gamma'_2$ is cyclic.
      \item  $\dim({\widetilde{M}}_x^{S^1})=3$. Let $N:={\widetilde{M}}_x^{S^1}$. 
      By Lemma \ref{lem:$T^2$-action with fixed point}, we may assume that ${\widetilde{M}}^{T^2}=\emptyset$
      and hence that there exists $S_1^1\subseteq T^2$ such that $N^{S_1^1}=\emptyset$. We apply Lemma \ref{lem:Z_p} 
      to the action of $S_1^1$ on $N$. For this, we use the Smith-Gysin sequence to calculate 
      $\chi(\bar{N},N^{S_1^1})=\chi(\bar{N},\emptyset)=2$. Since $\Gamma'_2$ has odd order, Lemma \ref{lem:Z_p}
      (see also Remark \ref{remark:Z_p trivial}) implies that $\Gamma'_2$ does not contain a copy 
      of ${\mathbb{Z}}_p\times{\mathbb{Z}}_p$ for any $p$. By Theorem \ref{thm:Wolf}, every Sylow subgroup 
      of $\Gamma'_2$ is cyclic. Hence by Theorem \ref{thm:Burnside}, $\Gamma'_2\in{\mathcal{C}}_d$ for some $d\geq 1$. 
      Lemma \ref{lem:Z_p} then implies that $d=1$ and hence that $\Gamma'_2$ is cyclic by Remark \ref{remark:C_d}. 
      \end{itemize}
Therefore, $\Gamma'_1$, and hence $\Gamma_1$, has a cyclic subgroup of index dividing three. 
Since $[\Gamma:\Gamma_1]\leq 2$, we are done.
\end{proof}

\subsection{Five-dimensional fixed point set}\label{sec:5}

In this section, we assume that there exists some $S^1\subseteq T^2$ whose fixed point set is five-dimensional.
Note that by Lemma \ref{lem:dimension 1 and 3}, we only need to discuss the case in which all components 
of ${\widetilde{M}}^{S^1}$ are five-dimensional. Since each component of ${\widetilde{M}}^{S^1}$ is a totally geodesic 
and hence positively curved submanifold of $\widetilde{M}$, the first Betti number of each component of the fixed point set ${\widetilde{M}}^{S^1}$ is zero. In addition, by Thereom \ref{thm:Allday}, the sum of the Betti numbers 
of ${\widetilde{M}}^{S^1}$ is at most six. Moreover, ${\widetilde{M}}^{S^1}$ cannot have the same rational cohomology 
ring as $(\mathbb{CP}^1\times{\mathbb{S}}^3)\#(\mathbb{CP}^1\times{\mathbb{S}}^3)$ 
since otherwise we will have $\sum_i b_i({\widetilde{M}}^{S^1};\mathbb{Q})=\sum_i b_i(\widetilde{M};\mathbb{Q})$. 
But this contradicts Theorem \ref{thm:Allday} because the rational cohomology ring  
of $(\mathbb{CP}^1\times{\mathbb{S}}^3)\#(\mathbb{CP}^1\times{\mathbb{S}}^3)$ has four generators. 
Altogether, we get that each component of ${\widetilde{M}}^{S^1}$ has the rational cohomology of either $\s^5$
or $\mathbb{CP}^1\times\s^3$. Therefore, the rational cohomology ring of ${\widetilde{M}}^{S^1}$ must be the same 
as that of one of the following spaces:
\begin{equation}\label{eq:cohomology ring}
\underbrace{{\mathbb{S}}^5\sqcup\ldots\sqcup{\mathbb{S}}^5}_{\text{k times}}~(1\leq k\leq 3)\hspace{0.7cm},\hspace{0.7cm}\mathbb{CP}^1\times{\mathbb{S}}^3\hspace{0.7cm},\hspace{0.7cm}{\mathbb{S}}^5\sqcup(\mathbb{CP}^1\times{\mathbb{S}}^3).
\end{equation}

\begin{lemma}\label{lem:dim 5}
For $M^{13}$ and $T^2$ as in Theorem \ref{thm:index 18}, if there exists $S^1\subseteq T^2$ 
such that the fixed point set ${\widetilde{M}}^{S^1}$ is five-dimensional, then $\Gamma:=\pi_1(M)$ has a cyclic subgroup 
of index dividing $18$ or $27$.
\end{lemma}

\begin{proof}
By the discussion before Lemma \ref{lem:dim 5}, we may assume that ${\widetilde{M}}^{S^1}$ 
has the same rational cohomology as one of the five spaces in (\ref{eq:cohomology ring}). We break the proof into cases:
        \begin{itemize}
        \item ${\widetilde{M}}^{S^1}$ has a component $N_1$ with the rational cohomology 
         of $\mathbb{CP}^1\times{\mathbb{S}}^3$. As in Lemma \ref{lem:dimension 1 and 3}, let $\Gamma_1\subseteq\Gamma$ 
         be the subgroup of index at most two that acts trivially on $H^*(\widetilde{M};\Q)$ and write $\Gamma_1$
         as $\Gamma_1\cong{\mathbb{Z}}_{2^a}\times\Gamma'_1$ for some $a\geq 0$ and some group $\Gamma'_1$ 
         of odd order. Since the action of $\Gamma'_1$ on $\widetilde{M}$ commutes with the $S^1$-action, $\Gamma'_1$ acts 
         on ${\widetilde{M}}^{S^1}$. Since ${\widetilde{M}}^{S^1}$ has at most one component with the rational cohomology 
         of $\mathbb{CP}^1\times{\mathbb{S}}^3$, $\Gamma'_1$ acts on $N_1$. 
         By Lemma \ref{lem:$T^2$-action with fixed point}, we may assume that ${\widetilde{M}}^{T^2}=\emptyset$ and  
         hence that there exists $S_2^1\subseteq T^2$ such that $N_1^{S_2^1}=\emptyset$. Let $\bar{N_1}:={N_1}/{S_2^1}$. 
         By applying the Smith-Gysin sequence to the pair $(\bar{N_1},N_1^{S_2^1})=(\bar{N_1},\emptyset)$, 
         it follows that $H^i(\bar{N_1};\Q)$ is isomorphic to $\Q$ for $i\in\{0, 4\}$, isomorphic to $\Q^2$ for $i=2$,
         and trivial otherwise.
         \par
         Let $\alpha\in\Gamma'_1$. We claim that ${\text{Lef}}(\alpha;\bar{N_1})$ equals $1$ or $4$. 
         In order to prove this, note that since $\alpha$ has odd order, the induced action of $\alpha$ on $H^i(\bar{N_1};\Q)$ 
         is trivial for $i\in\{0, 4\}$. Now, let $\lambda_1$ and $\lambda_2$
         denote the eigenvalues of $\alpha^*:H^2(\bar{N_1};\Q)\to H^2(\bar{N_1};\Q)$. Since $\alpha$ has finite order,
         $\lambda_1$ and $\lambda_2$ are roots of unity, so $\lambda_1+\lambda_2\in[-2,2]$. Moreover, the Lefschetz number 
         and hence $\lambda_1+\lambda_2$ is an integer. Therefore, $\lambda_1+\lambda_2\in\{-2,-1,0,1,2\}$.
         In addition, $\lambda_1$ and $\lambda_2$ have odd orders since $\alpha$ has odd order. 
         In particular, $\lambda_i\neq -1$, so $\lambda_1+\lambda_2\neq -2$. Similarly, if $\lambda_1\neq 1$, 
         then it is complex and hence $\lambda_2=\bar{\lambda}_1$ and $\lambda_1+\lambda_2=2Re(\lambda_1)$. 
         Therefore, if $\lambda_1+\lambda_2\in\{0, 1\}$, then $\lambda_1$ has order $4$ or $6$, a contradiction. 
         Hence $\lambda_1+\lambda_2\in\{-1,2\}$. This proves the claim.
         \par
         Lemma \ref{lem:Z_p} now implies that $\Gamma'_1$ does not contain 
         ${\mathbb{Z}}_p\times{\mathbb{Z}}_p$ for all $p$. Hence by Theorem \ref{thm:Wolf}, all Sylow subgroups 
         of $\Gamma'_1$ are cyclic. Theorem \ref{thm:Burnside} and Lemma \ref{lem:Z_p} then imply that 
         $\Gamma'_1\in{\mathcal{C}}_d$, where $d=1$. Therefore, $\Gamma'_1$, and hence $\Gamma_1$, is cyclic.
         This means that $\Gamma$ has a cyclic subgroup of index at most two.
        \item ${\widetilde{M}}^{S^1}$ is a rational cohomology $5$-sphere. Let $N_2:={\widetilde{M}}^{S^1}$. 
         Note that $\Gamma$ acts on $N_2$. Moreover, $\Gamma$ acts trivially on the rational cohomology of $N_2$ 
         by Weinstein's theorem. Therefore, we can apply Theorem \ref{thm:Davis} to conclude that  
         $\Gamma\cong{\mathbb{Z}}_{2^b}\times\Gamma'$, where $b\geq 0$ and $\Gamma'$ is a group of odd order. 
         As in the previous case, we can choose $S_3^1\subseteq T^2$ such that $N_2^{S_3^1}=\emptyset$ and apply 
         the Smith-Gysin sequence to the pair $(\bar{N_2},N_2^{S_3^1})=(\bar{N_2},\emptyset)$ to get $\chi(\bar{N_2})=3$.
         Lemma \ref{lem:Z_p} then implies that $\Gamma'$ does not contain a copy of ${\mathbb{Z}}_p\times{\mathbb{Z}}_p$ 
         for $p\neq 3$. This, together with Theorem \ref{thm:Wolf}, implies that all Sylow $p$-subgroups of $\Gamma'$ are cyclic 
         for $p\neq 3$. Now, we claim that the Sylow $3$-subgroup of $\Gamma'$ is either cyclic or isomorphic 
         to one of ${\mathbb{Z}}_3\times{\mathbb{Z}}_3$ or ${\mathbb{Z}}_9\rtimes{\mathbb{Z}}_3$. The idea is to prove 
         that $\Gamma'$ does not contain a copy of ${\mathbb{Z}}_3\times{\mathbb{Z}}_3\times{\mathbb{Z}}_3$, 
         $U(3,3)$, or ${\mathbb{Z}}_9\times{\mathbb{Z}}_3$. Proposition \ref{prop:3-groups} then implies 
         the claim. By Lemma \ref{lem:Z_p}, in order to prove that $\Gamma'$ does not contain a copy 
         of ${\mathbb{Z}}_3\times{\mathbb{Z}}_3\times{\mathbb{Z}}_3$, $U(3,3)$, 
         or ${\mathbb{Z}}_9\times{\mathbb{Z}}_3$, it suffices to find a normal subgroup of order three of each of these groups 
         which is not strictly contained in a cyclic subgroup. Existence of such a subgroup  is obvious in the case 
         of ${\mathbb{Z}}_3\times{\mathbb{Z}}_3\times{\mathbb{Z}}_3$ or ${\mathbb{Z}}_9\times{\mathbb{Z}}_3$. 
         As for $U(3,3)$, it has center isomorphic to ${\mathbb{Z}}_3$. This subgroup is normal and is not strictly contained 
         in a larger cyclic subgroup because every non-trivial element of $U(3,3)$ has order three.
         \par
         Now, we have a group $\Gamma'$ of odd order such that its Sylow $p$-subgroups are cyclic for $p\neq 3$ 
         and its Sylow $3$-subgroup is either cyclic, ${\mathbb{Z}}_3\times{\mathbb{Z}}_3$, 
         or ${\mathbb{Z}}_9\rtimes{\mathbb{Z}}_3$. Let $P_3$ be a Sylow $3$-subgroup of $\Gamma'$. 
         If $P_3$ is cyclic, then $\Gamma'\in{\mathcal{C}}_d$ for some $d\geq 1$ by Theorem \ref{thm:Burnside}. 
         If not, then we apply Theorem \ref{thm:normal complement}. Hence $\Gamma'$ 
         can be written in the form of $P_3N$ for some $N\trianglelefteq\Gamma'$ such that $P_3\cap N=\{1\}$. 
         Letting $\Gamma''_2:={\mathbb{Z}}_3N$ or $\Gamma''_2:={\mathbb{Z}}_9N$, depending on whether   
         $P_3\cong{\mathbb{Z}}_3\times{\mathbb{Z}}_3$ or $P_3\cong{\mathbb{Z}}_9\rtimes{\mathbb{Z}}_3$,
         we get an index three subgroup $\Gamma''_2$ of $\Gamma'$ such that all Sylow subgroups 
         of $\Gamma''_2$ are cyclic. By Theorem \ref{thm:Burnside}, it follows that $\Gamma''_2\in{\mathcal{C}}_d$ 
         for some $d$. The discussion above shows that either $\Gamma'$ or $\Gamma''_2$ has a cyclic subgroup of index $d$. 
         Now, Lemma \ref{lem:Z_p} implies that $d$ divides three. Therefore, $\Gamma'$, and hence $\Gamma$, has a cyclic    
         subgroup of index dividing nine, as required.
        \item ${\widetilde{M}}^{S^1}$ is the disjoint union of two rational cohomology $5$-spheres. In this case, 
         let ${\widetilde{M}}_x^{S^1}$ denote one of the two components of ${\widetilde{M}}^{S^1}$ 
         and let $\Gamma_3$ be the subgroup of $\Gamma$ of index at most two which acts on ${\widetilde{M}}_x^{S^1}$.
         Replacing $\Gamma$ by $\Gamma_3$ in the argument for the previous case, it follows that $\Gamma$ has a cyclic     
         subgroup of index dividing $18$.
        \item ${\widetilde{M}}^{S^1}$ is the disjoint union of three rational cohomology $5$-spheres. 
         Let ${\widetilde{M}}_x^{S^1}$ denote one of the components of ${\widetilde{M}}^{S^1}$
         and let $\Gamma_3\subseteq\Gamma$ be the subgroup of index at most three that acts 
         on ${\widetilde{M}}_x^{S^1}$. We again argue as in the second case, replacing $\Gamma$ by $\Gamma_3$
         to conclude that $\Gamma$ has a cyclic subgroup of index dividing $18$ or $27$.\qedhere
        \end{itemize}
\end{proof}

\subsection{Conclusion of proof}\label{sec:conclusion}

Let $M$ and $T^2$ be as in Theorem \ref{thm:index 18}. As remarked at the beginning of this section, 
we may lift the $T^2$-action to the universal cover of $M$. By Theorem \ref{thm:Berger}, we may choose $S^1\subseteq T^2$
such that ${\widetilde{M}}^{S^1}\neq\emptyset$. If ${\widetilde{M}}^{S^1}$ has a component of dimension one or three, 
then we are done by Lemma \ref{lem:dimension 1 and 3}. In particular, we are done if $\dim({\widetilde{M}}^{S^1})\leq 3$. 
In addition, if $\dim({\widetilde{M}}^{S^1})=5$, then the proof follows by Lemma \ref{lem:dim 5}. 
Recall also that the dimension of ${\widetilde{M}}^{S^1}$ is at most seven by Lemma \ref{lem:totally geodesic}.
We may assume therefore that $\dim({\widetilde{M}}^{S^1})=7$. 

Note that by Frankel's theorem (see \cite{Frankel}), ${\widetilde{M}}^{S^1}$ cannot have more than one component 
of dimension seven. Moreover, the seven-dimensional component ${\widetilde{M}}_x^{S^1}$ of ${\widetilde{M}}^{S^1}$
satisfies $b_2({\widetilde{M}}_x^{S^1};\Q)=1$ since the inclusion ${\widetilde{M}}_x^{S^1}\embedded\widetilde{M}$
is $3$-connected by Theorem \ref{thm:connectedness lemma}. By Poincar\'e duality, ${\widetilde{M}}_x^{S^1}$
has total Betti number at least four. Suppose for a moment that ${\widetilde{M}}^{S^1}$ is not connected.
Let ${\widetilde{M}}_y^{S^1}$ denote another component. By Theorem \ref{thm:Allday}, ${\widetilde{M}}_y^{S^1}$
is a rational $\s^1$, $\s^3$, or $\s^5$ because $\widetilde{M}$ has total Betti number six. 
It follows by Lemma \ref{lem:dimension 1 and 3} and by the second case in the proof Lemma \ref{lem:dim 5}
that $\pi_1(M)$ has a cyclic subgroup of index $1$, $2$, $3$, $6$, or $9$. 
Therefore, we may assume that ${\widetilde{M}}^{S^1}$ is connected.
 
\begin{lemma}\label{lem:dimension 7}
Suppose that $T^2$ acts effectively and isometrically on a closed, positively curved manifold $M^{13}$ whose universal cover 
is a rational cohomology Bazaikin space. If there exists $S^1\subseteq T^2$ such that ${\widetilde{M}}^{S^1}$
is connected and seven-dimensional, then $\pi_1(M)$ is cyclic, ${\widetilde{M}}^{T^2}\neq\emptyset$,
or there exists another circle $S_1^1\subseteq T^2$ such that ${\widetilde{M}}^{S_1^1}$ is non-empty and has dimension 
at most five.
\end{lemma}

\begin{proof}
Let $p:\widetilde{M}\to M$ be the universal covering map. Note that $p({\widetilde{M}}^{S^1})=M^{S^1}$,
so $M^{S^1}$ is connected and seven-dimensional. 
\par
If the action of $S^1$ on $M$ does not have any non-trivial 
finite isotropy groups, then $\pi_1(M)$ is cyclic (see \cite[Theorem C]{Rong2005}) and we are done. 
Hence we may assume that this action has a non-trivial finite isotropy group $S_q^1$. Let $N:=M_q^{S_q^1}$. 
\par
We claim that $N\cap  M^{S^1}$ is empty. Indeed, if $q'\in N\cap  M^{S^1}$, 
then we have $M^{S^1}\subseteq M_{q'}^{S_q^1}=M_q^{S_q^1}$. 
Now, $q\notin M^{S^1}$ since the isotropy group $S_q^1$ is finite, so this inclusion is strict. But then the dimension 
of $M_q^{S_q^1}$ would be greater than seven, so we have a contradiction to Lemma \ref{lem:totally geodesic}.
\par
Next we claim that there exists a circle $S_1^1\subseteq T^2$ such that $N^{S_1^1}\neq\emptyset$. 
In order to prove this, note that $T^2$ acts on $N$. Let $H$ denote the kernel of this action. 
If $H$ is not finite, then there exists $S_1^1\subseteq H$ and we are done. 
If $H$ is finite, then the claim follows by applying Theorem \ref{thm:Berger} to the action of ${T^2}/H\cong T^2$ on $N$.
Therefore, we can choose $q_1\in N^{S_1^1}$.
\par
We claim that $S_1^1\neq S^1$. If instead $S_1^1=S^1$, then $q_1\in M^{S^1}$ since $q_1$ is fixed by $S_1^1$. 
But $q_1\in N$, so $q_1\in N\cap  M^{S^1}$, a contradiction.
\par
To conclude the proof, we consider two cases. If $\dim(M^{S_1^1})\leq 5$, then $\dim({\widetilde{M}}^{S_1^1})\leq 5$
and we are done. If not, then $M^{S^1}\cap M^{S_1^1}\neq\emptyset$ by Frankel's theorem. 
In this case, $T^2$ has a fixed point on $M$ and hence on $\widetilde{M}$.
\end{proof}

Lemma \ref{lem:dimension 7}, together with Lemmas \ref{lem:$T^2$-action with fixed point}, \ref{lem:dimension 1 and 3}, 
and \ref{lem:dim 5}, completes the proof of Theorem \ref{thm:index 18} for rational cohomology Bazaikin spaces.

\section{Proof of theorem \ref{thm:index 18} for mod $3$ cohomology Bazaikin spaces}\label{sec:Z_3}

In this section, we prove Theorem \ref{thm:index 18} for mod $3$ cohomology Bazaikin spaces. The idea here is to find a subgroup $\Gamma_3\subseteq\Gamma$ of index at most three which does not contain a copy of ${\mathbb{Z}}_p\times{\mathbb{Z}}_p$
for any $p$. Recall that by a result of Smith, ${\mathbb{Z}}_p\times{\mathbb{Z}}_p$ cannot act freely on a mod $p$ cohomology sphere (see \cite{Smith}). Heller proved that ${\mathbb{Z}}_p^3$ cannot act freely on a manifold $M$ 
with $H^*(M;{\mathbb{Z}}_p)\cong H^*({\mathbb{S}}^m\times{\mathbb{S}}^n;{\mathbb{Z}}_p)$, where $n>m$
(see \cite{Heller}, cf. \cite[Example 3.10.17]{AlldayPuppe}). 
We refine Heller's argument in a special case to prove the following lemma:

\begin{lemma}\label{lem:free actions on product of spheres}
If $M$ is a smooth manifold such that 
$H^*(M;{\mathbb{Z}}_p)\cong H^*({\mathbb{S}}^2\times{\mathbb{S}}^3;{\mathbb{Z}}_p)$ for some odd prime $p$,
then ${\mathbb{Z}}_p\times{\mathbb{Z}}_p$ cannot act freely on $M$. 
\end{lemma}

\begin{proof}
Suppose by contradiction that $G={\mathbb{Z}}_p\times{\mathbb{Z}}_p$ acts freely on $M$ and consider the Serre spectral sequence associated to the Borel fibration $M\to M_G\to BG$. Since $p$ is odd and $\dim H^i(M;\Z_p)\leq 1$ for all $i$, 
$G$ acts trivially on $H^*(M;\Z_p)$ and the spectral sequence exists.
\par
Since $G$ acts freely on $M$, it follows from \cite[Proposition 3.10.9]{AlldayPuppe}) 
that $H^*(M_G;{\mathbb{Z}}_p)\cong H^*(M/G;{\mathbb{Z}}_p)$.
By \cite[Lemma 3.10.16]{AlldayPuppe}, this means that $H^i(M_G;{\mathbb{Z}}_p)=0$ for all $i>5$. 
In particular, $H^6(M_G;{\mathbb{Z}}_p)=0$ and hence $E_{\infty}^{6,0}$ is trivial. Observe that for the differential map $d_r:E_r^{m,n}\to E_r^{6,0}$, $m$ and $n$ satisfy $m+n=5$ and $n\geq 1$. Moreover, $E_2^{m,n}=0$ for $n=1$ or $4$.
Hence the only non-trivial differentials which can kill elements of $E_2^{6,0}$ come from $E_3^{3,2}$, $E_4^{2,3}$ 
and $E_6^{0,5}$. Since $H^*(BG;{\mathbb{Z}}_p)={\mathbb{Z}}_p[t_1,t_2]\tensor\wedge(s_1,s_2)$, 
where each $t_i$ has degree two and each $s_i$ has degree one, it follows that the set $\{t_1^it_2^j:i+j=3\}\cup\{s_1s_2t_1^it_2^j:i+j=2\}$ forms a basis for $E_2^{6,0}$ and hence $\dim E_2^{6,0}=7$. Moreover,
\begin{align*}
\dim E_3^{3,2}+\dim E_4^{2,3}+\dim E_6^{0,5} & \leq\dim E_2^{3,2}+\dim E_2^{2,3}+\dim E_2^{0,5}\\
 & =\dim H^3(BG;{\mathbb{Z}}_p)+\dim H^2(BG;{\mathbb{Z}}_p)+\dim H^0(BG;{\mathbb{Z}}_p)\\
 & =4+3+1=8.
\end{align*}
To get a contradiction, it suffices to find at least two basis elements in the groups $E_2^{3,2}$, $E_2^{2,3}$, and $E_2^{0,5}$
that do not survive to the appropriate page and kill any elements of $E_2^{6,0}$. 
In order to find such basis elements, we need to analyze the differentials. Fix generators $x\in H^3(M;{\mathbb{Z}}_p)$ 
and $y\in H^2(M;{\mathbb{Z}}_p)$. We break the proof into two cases:
         \begin{itemize}
         \item The differential $d_2:E_2^{1,3}\to E_2^{3,2}$ is non-zero. In this case, the differential 
         $d_2:E_2^{0,3}\to E_2^{2,2}$ is non-zero because $\{s_1x, s_2x\}$ forms a basis for $E_2^{1,3}$ and $d_2(s_i)=0$.
         Note that $d_2(x)$ is of the form $a_1t_1y+a_2t_2y+a_3s_1s_2y$ for some $a_i\in{\mathbb{Z}}_p$.
         Note also that at least one of $a_1$ or $a_2$ is non-zero since otherwise $d_2$ vanishes on $E_2^{1,3}$.
         Therefore, $d_2(s_ix)=-a_1s_it_1y-a_2s_it_2y$ and hence $d_2:E_2^{1,3}\to E_2^{3,2}$ is injective. This means
         that two of the basis elements of $E_2^{3,2}$ do not survive to the $E_3$ page and hence cannot kill any elements 
         of $E_3^{6,0}$.
         \item The differential $d_2:E_2^{1,3}\to E_2^{3,2}$ is zero. In this case, 
         $E_3^{3,2}\cong E_2^{3,2}\cong H^3(BG;{\mathbb{Z}}_3)$. Since $y$ survives to the $E_3$ page, the set 
         $\{\bar{s}_1\bar{t}_1\bar{y},\bar{s}_1\bar{t}_2\bar{y},\bar{s}_2\bar{t}_1\bar{y},\bar{s}_2\bar{t}_2\bar{y}\}$
         forms a basis for $E_3^{3,2}$, where the bars denote the images of elements from the $E_2$ page in the $E_3$ page. 
         Without loss of generality, we may assume that $d_3(\bar{y})\neq 0$ since otherwise $d_3:E_3^{3,2}\to E_3^{6,0}$        
         would be trivial and so $E_3^{3,2}$ cannot kill any elements of $E_3^{6,0}$. 
         Now, we claim that after possibly a change of basis in $E_3^{1,0}$ and $E_3^{2,0}$, $d_3(\bar{y})$ 
         can be written in the form of $\bar{s}_1\bar{t}_1$ or $\bar{s}_1\bar{t}_1+\bar{s}_2\bar{t}_2$. 
         In order to prove this, observe that     
         $d_3(\bar{y})=b_1\bar{s}_1\bar{t}_1+b_2\bar{s}_1\bar{t}_2+b_3\bar{s}_2\bar{t}_1+b_4\bar{s}_2\bar{t}_2$, 
         where $b_i\in{\mathbb{Z}}_p$. Moreover, since $d_3(\bar{y})\neq 0$, we may assume that $b_1\neq 0$.          
         Letting $s'_1=b_1\bar{s}_1+b_3\bar{s}_2$ implies $d_3(\bar{y})$ is of the form 
         $s'_1\bar{t}_1+c_1s'_1\bar{t}_2+c_2\bar{s}_2\bar{t}_2$ for some $c_1, c_2\in{\mathbb{Z}}_p$. Now, letting
         $t'_1=\bar{t}_1+c_1\bar{t}_2$ implies $d_3(\bar{y})=s'_1t'_1+c_2\bar{s}_2\bar{t}_2$. Finally, either $c_2=0$ 
         or $c_2\neq 0$. In the former case, $d_3(\bar{y})$ is of the form $\bar{s}_1\bar{t}_1$ and in the latter case, by letting 
         $s'_2=c_2\bar{s}_2$, it follows that $d_3(\bar{y})$ is of the form $\bar{s}_1\bar{t}_1+\bar{s}_2\bar{t}_2$. 
         \par
         The claim proves that the kernel of $d_3:E_3^{3,2}\to E_3^{6,0}$ is at least 1-dimensional and so not all basis     
         elements of $E_3^{3,2}$ can contribute to killing elements of $E_3^{6,0}$. Now, we only need to find one more basis 
         element in another group that does not kill any elements of $E_2^{6,0}$.  
         If $d_2(x)=0$, then $x$ survives to the $E_3$ page and so we have
         $d_3(\bar{x}\bar{y})=d_3(\bar{x})\bar{y}-\bar{x}d_3(\bar{y})\neq 0$. This means that $E_3^{0,5}$ does not survive 
         to the $E_4$ page and hence it cannot kill any elements of $E_6^{6,0}$. If instead $d_2(x)\neq 0$,
         then $d_2:E_2^{2,3}\to E_2^{4,2}$ would be non-zero. This means that at least one of the basis elements 
         of $E_2^{2,3}$ does not survive to the $E_3$ page and so it cannot kill any elements of $E_4^{6,0}$.\qedhere
        \end{itemize}
\end{proof}

We now proceed to the proof of Theorem \ref{thm:index 18}. As the proof for rational cohomology Bazaikin spaces shows, index $18$ or $27$ would possibly arise only when the fixed point set of the circle action is the disjoint union of either two or three rational cohomology $5$-spheres. Hence we only need to discuss those two cases.

First, suppose that ${\widetilde{M}}^{S^1}={\widetilde{M}}_x^{S^1}\sqcup{\widetilde{M}}_y^{S^1}$, where each 
component is a rational cohomology $5$-sphere. Consider ${\mathbb{Z}}_3\subseteq S^1$. 
We may assume that ${\widetilde{M}}_x^{S^1}={\widetilde{M}}_x^{{\mathbb{Z}}_3}$ since otherwise we get 
strict inclusions ${\widetilde{M}}_x^{S^1}\subsetneq{\widetilde{M}}_x^{{\mathbb{Z}}_3}\subsetneq\widetilde{M}$.
Lemma \ref{lem:totally geodesic} then implies that $\dim({\widetilde{M}}_x^{{\mathbb{Z}}_3})=7$. 
By Corollary \ref{cor:codimension 2}, ${\widetilde{M}}_x^{{\mathbb{Z}}_3}$ is a homotopy sphere.
Since the inclusion ${\widetilde{M}}_x^{{\mathbb{Z}}_3}\embedded\widetilde{M}$ is $3$-connected by Theorem
\ref{thm:connectedness lemma}, we have a contradiction.
Similarly, ${\widetilde{M}}_y^{S^1}={\widetilde{M}}_y^{{\mathbb{Z}}_3}$. Moreover,
$$\sum_i\dim H^i({\widetilde{M}}^{{\mathbb{Z}}_3};{\mathbb{Z}}_3)\leq\sum_i\dim H^i(\widetilde{M};{\mathbb{Z}}_3)\leq 10.$$
Together with Poincar\'e duality, this implies that at least one of the fixed point set components, say ${\widetilde{M}}_x^{S^1}$, has the mod $3$ cohomology of either ${\mathbb{S}}^5$ or ${\mathbb{S}}^2\times{\mathbb{S}}^3$
(note that ${\widetilde{M}}_x^{S^1}$ is positively curved, so it has vanishing first integral Betti number.
This, together with the universal coefficient theorem, implies that ${\widetilde{M}}_x^{S^1}$ cannot have the mod $3$ cohomology of ${\mathbb{S}}^1\times{\mathbb{S}}^4$). 
Let $\Gamma_3$ denote the subgroup of $\Gamma$ of index at most two which acts on ${\widetilde{M}}_x^{S^1}$. 
By Theorem \ref{thm:Davis}, we have $\Gamma_3\cong{\mathbb{Z}}_{2^c}\times\Gamma'_3$, where $c\geq 0$ 
and $\Gamma'_3$ is a group of odd order.
As the proof of Theorem \ref{thm:index 18} for rational cohomology Bazaikin spaces shows, $\Gamma'_3$ 
does not contain a copy of ${\mathbb{Z}}_p\times{\mathbb{Z}}_p$ for $p\neq 3$. In addition, Lemma 
\ref{lem:free actions on product of spheres}, together with the fact that ${\mathbb{Z}}_p\times{\mathbb{Z}}_p$ 
cannot act freely on a mod $p$ cohomology sphere, implies that $\Gamma'_3$ does not contain 
${\mathbb{Z}}_3\times{\mathbb{Z}}_3$ . Hence all Sylow subgroups of $\Gamma'_3$ are cyclic
and $\Gamma'_3\in{\mathcal{C}}_d$ for some $d\geq 1$. Our calculations from the previous section implies that $d$ 
divides three. This means that $\Gamma_3$ has a cyclic subgroup of index dividing three and hence $\Gamma$ has a cyclic subgroup of index dividing six. 
\par
The proof for the case where ${\widetilde{M}}^{S^1}$ is the disjoint union of three rational cohomology $5$-spheres 
is similar except that here the bound on $\dim H^*({\widetilde{M}}^{{\mathbb{Z}}_3};{\mathbb{Z}}_3)$ 
implies that at least one of the components of ${\widetilde{M}}^{S^1}$ is a mod $3$ cohomology $\s^5$. 
Moreover, in this case, the index of $\Gamma_3$ in $\Gamma$ is at most three and hence we get that $\Gamma$ 
has a cyclic subgroup of index $1, 2, 3, 6$, or $9$.

\end{document}